\documentclass{amsart}
\usepackage{amssymb}
\usepackage[all,ps,cmtip]{xy}

\newtheorem{theorem}{Theorem}[section]
\newtheorem{lemma}[theorem]{Lemma}

\newtheorem{corollary}[theorem]{Corollary}

\newtheorem{claim}[theorem]{Claim}

\theoremstyle{definition}
\newtheorem{definition}[theorem]{Definition}

\theoremstyle{remark}

\numberwithin{equation}{section}
\DeclareMathOperator{\Alb}{Alb}

\DeclareMathOperator{\Pic}{Pic}

\DeclareMathOperator{\codim}{codim}

\DeclareMathOperator{\Spec}{Spec}

\begin{document}

\title{Pluricanonical maps of varieties of Albanese fiber dimension two}

\author{Hao Sun}

\address{Department of Mathematics, Shanghai Normal University, Shanghai 200234, People's Republic of China}

\address{Department of Mathematics, Huazhong Normal University, Wuhan 430079, People's Republic of China}

\email{hsun@mail.ccnu.edu.cn, hsunmath@gmail.com}
\thanks{This work is supported by NSFC and the Mathematical Tianyuan Foundation of China.}



\date{November 16, 2012}


\begin{abstract}
In this paper we prove that for any smooth projective variety of
Albanese fiber dimension two and of general type, the $6$-canonical
map is birational. And we also show that the $5$-canonical map is
birational for any such variety with some geometric restrictions.
\end{abstract}

\maketitle
\section{Introduction}
Let $X$ be a smooth complex projective irregular variety of general
type, i.e., variety of general type with $q(X)>0$. We define the
Albanese fiber dimension of $X$ to be $e= \dim X-\dim a(X)$, where
$a: X\rightarrow \Alb(X)$ is the Albanese map. Recently, the
birationality of the $n$-th pluricanonical map $\varphi_{|nK_X|}$ of
$X$ has attracted a lot of attention.

When $e=0$, i.e., $X$ is of maximal Albanese dimension, it was shown
by Chen and Hacon \cite{CH1, CH2} that $\varphi_{|6K_X|}$ is
birational. This result was improved by Jiang, Lahoz and Tirabassi
\cite{JLT}, showing that $\varphi_{|3K_X|}$ is birational. When $e=
1$ or $2$, Chen and Hacon \cite{CH2} proved that
$\varphi_{|(e+5)K_X|}$ is a birational map. Recently, Jiang and the
author \cite{JS} showed that $\varphi_{|4K_X|}$ is birational if
$e=1$. The main results of this paper are some improvements of the
result of Chen and Hacon in the case of Albanese fiber dimension
two:

\begin{theorem}\label{theorem1.1}
Let $X$ be a smooth complex projective variety of Albanese fiber
dimension two and general type. Then the $6$-canonical map
$\varphi_{|6K_X|}$ is birational.
\end{theorem}

\begin{theorem}\label{theorem1.2}
Let $X$ be a smooth complex projective variety of Albanese fiber
dimension two and general type. If the translates through $0$ of all
components of $V^0(\omega_X)$ generate $\Pic^0(X)$, then
$\varphi_{|5K_X|}$ is birational.
\end{theorem}

By Theorem \ref{theorem1.1}, we can immediately obtain a recent
result of Chen, Chen and Jiang:

\begin{corollary}\cite[Theorem 1.1]{CCJ}
Let $V$ be a smooth complex projective irregular $3$-fold of general
type. Then $\varphi_{|6K_X|}$ is birational.
\end{corollary}

\subsection*{Acknowledgments}The author would like to thank Jungkai A. Chen, Zhi Jiang
and Lei Zhang for various comments and useful discussions. This work
was done while the author was visiting Emmy Noether Research
Institute for Mathematics. He is grateful to this institution for
hospitality.

\section{Definitions and lemmas}
In this section, we recall some notion and useful lemmas.

\begin{definition}
Let $\mathcal{F}$ be a coherent sheaf on a smooth projective variety
$Y$.
\begin{enumerate}
\item The $i$-th cohomological support loci of $\mathcal{F}$ is
$$V^i(\mathcal{F})=\{\alpha\in\Pic^0(Y)|h^i(\mathcal{F}\otimes\alpha)>0\}.$$

\item We say $\mathcal{F}$ is $IT^0$ if $H^i(\mathcal{F}\otimes\alpha)=0$ for all $i\geq1$ and all $\alpha\in\Pic^0(Y)$.

\item $\mathcal{F}$ is called $M$-regular if $\codim
V^i(\mathcal{F})>i$ for every $i>0$ and $Y$ is an abelian variety.

\item We say $\mathcal{F}$ is continuously globally generated at $y\in Y$ (in brief CGG at $y$)
if the nature map $$\bigoplus_{\alpha\in U}
H^0(\mathcal{F}\otimes\alpha)\otimes{\alpha}^{\vee}\rightarrow
\mathcal{F}\otimes \mathbb{C}(y)$$ is surjective for any non-empty
open subset $U\subset{\Pic}^0(Y)$.

\item $\mathcal{F}$ is said to have no essential base point at $y\in Y$
if for any surjective map $\mathcal{F}\rightarrow \mathcal{O}_y$,
there is a non-empty open subset $U\subset{\Pic}^0(Y)$ such that for
all $\alpha\in U$, the induced map
$H^0(\mathcal{F}\otimes\alpha)\rightarrow
H^0(\mathcal{O}_y\otimes\alpha)$ is surjective.
\end{enumerate}
\end{definition}

\begin{lemma}\label{lemma2.2}
Let $\pi:X\rightarrow Y$ be a double covering branched along a
reduced divisor $B\in|2L|$, where $X$ is a projective variety, $Y$
is a smooth projective variety and $L$ is a divisor on $Y$. Let $D$
be a divisor on $Y$. Then $|\pi^*D|$ induces a birational map if and
only if $|D|$ induces a birational map of $Y$ and $H^0(Y,
\mathcal{O}_Y(D-L))\neq0$.
\end{lemma}
\begin{proof}
We know that
$\pi_*\pi^*\mathcal{O}_Y(D)=\mathcal{O}_Y(D)\oplus\mathcal{O}_Y(D-L)$.
Hence we have
$$H^0(X,\mathcal{O}_X(\pi^*D))=H^0(Y,\mathcal{O}_Y(D))\oplus H^0(Y,\mathcal{O}_Y(D-L)).$$
The only-if-direction follows immediately from the above
isomorphism.

To prove the if-direction, we take an open affine subset $U\subset
Y$. Suppose that $U=\Spec R$ for some ring $R$. We know that
$\pi^{-1}(U)=\Spec R[z]/(z^2-f)$, where $f\in R$ is the local
defining equation of $B$. From the isomorphism of $R$-modules
$$R[z]/(z^2-f)\cong R\oplus Rz,$$ it follows that we can choose $s_1,\ldots, s_k, t_1,\ldots, t_k\in
R$ such that $1, s_1,\ldots, s_k$ is a basis of $H^0(D)$ and $t_1z,
\ldots, t_nz$ is a basis of $H^0(D-L)$. Let $y\in Y$ be a general
point. Since $H^0(D-L)\neq0$, we can assume that $f(y)\neq0$ and
$t_1(y)\neq0$.

Let $\{x_1, x_2\}$ be the preimage of $y$, where $x_1=(y,
\sqrt{f(y)})$ and $x_2=(y, -\sqrt{f(y)})$. Since $|\pi^*D|$
separates two general points on two distinct general fibers of
$\pi$, we can find a section $s_0+t_0z\in H^0(X,
\mathcal{O}_X(\pi^*D))$ vanishes along $y$, where $s_0\in H^0(D)$
and $t_0z\in H^0(D-L)$. Hence section
$$s_0+(t_0(y)-t_1(y))\sqrt{f(y)}+t_1z\in H^0(\pi^*D)$$ vanishes along $x_1$ but does not vanish along
$x_2$. It follows that $|\pi^*D|$ separates two points on a general
fiber of $\pi$. Therefore $|\pi^*D|$ induces a birational map.
\end{proof}

\begin{lemma}\label{lemma2.3}
Let $f: X\rightarrow W$ be a morphism between smooth projective
varieties with a general fiber $F$. Suppose that $\kappa(W)\geq0$
and $K_X$ is $W$-big, i.e., $sK_X\geq f^*L$ for some ample divisor
$L$ on $W$ and some integer $s\gg0$. Suppose further that
$h^0(mK_F)>0$ for some $m\geq2$. Then after replacing $X$ by an
appropriate birational model, there exist positive integers $b$, $c$
and there is a normal crossing divisor $$B\in|bc(m-1)K_X-f^*bM|$$
such that $\lfloor\frac{B}{bc}\rfloor|_F\leq \mathcal{B}_{m,F}$,
$\lfloor\frac{B}{bc}\rfloor\leq \mathcal{B}_{m,\alpha}$, for all
$\alpha\in\Pic^0(W)$. Here $\mathcal{B}_{m,F}$ (resp.
$\mathcal{B}_{m,\alpha}$) is the fixed part of $|mK_F|$ (resp.
$|mK_X+f^*\alpha|$), $M$ is a given nef and big divisor on $W$ and
$b$, $c$ are sufficiently large integers depending on $M$ and $K_X$.
\end{lemma}
\begin{proof}
See \cite[Lemma 2.5]{CH3}.
\end{proof}

\begin{lemma}\label{lemma2.4}
If $\mathcal{F}$ is a coherent sheaf on a smooth projective variety
$Y$ and $y$ is a point on $Y$, then $\mathcal{F}$ is CGG at $y$ if
and only if $\mathcal{F}$ has no essential base point at $y$.
\end{lemma}
\begin{proof}
Firstly, we assume that $\mathcal{F}$ is CGG at $y$. For any
surjective map $\mathcal{F}\rightarrow \mathcal{O}_y$, the induced
map $\mathcal{F}\otimes \mathbb{C}(y)\rightarrow \mathcal{O}_y$ is
also surjective. The definition of CGG implies that the composition
$$\bigoplus_{\alpha\in U}
H^0(\mathcal{F}\otimes\alpha)\otimes{\alpha}^{\vee}\rightarrow
\mathcal{F}\otimes \mathbb{C}(y)\rightarrow \mathcal{O}_y$$ is
surjective for any non-empty open subset $U\subset{\Pic}^0(Y)$. It
follows that for any non-empty open subset $U\subset{\Pic}^0(Y)$,
there is an $\alpha\in U$ such that the induced map
$H^0(\mathcal{F}\otimes\alpha)\rightarrow
H^0(\mathcal{O}_y\otimes\alpha)$ is surjective. By semi-continuity,
one sees that for a general $\alpha\in\Pic^0(Y)$, the induced map
$H^0(\mathcal{F}\otimes\alpha)\rightarrow
H^0(\mathcal{O}_y\otimes\alpha)$ is surjective. Thus $\mathcal{F}$
has no essential base point at $y$.

Conversely, suppose that $\mathcal{F}$ has no essential base point
at $y$. One can write $\mathcal{F}\otimes
\mathbb{C}(y)=\bigoplus_{i=1}^kV_i$, where $V_i\cong \mathbb{C}(y)$,
$i=1,2,\ldots, k$. Let $p_i:
\mathcal{F}\otimes\mathbb{C}(y)\rightarrow V_i$ be the canonical
projection and
$\varphi_i:\mathcal{F}\rightarrow\mathcal{F}\otimes\mathbb{C}(y)\rightarrow
V_i$ be the composition. Since $\mathcal{F}$ has no essential base
point at $y$, we know that there exists non-empty open subsets $U_i$
($1\leq i\leq k$) of $\Pic^0(Y)$ such that for any $\alpha\in U_i$
the induced map $\tilde{\varphi_i}:
H^0(\mathcal{F}\otimes\alpha)\rightarrow V_i\otimes\alpha$ is
surjective. For any non-empty open subset $U_0\subset{\Pic}^0(Y)$,
since $\cap_{i=0}^kU_i$ is also a non-empty open subset, the map
$$\bigoplus_{\alpha\in U_0}
H^0(\mathcal{F}\otimes\alpha)\otimes{\alpha}^{\vee}\longrightarrow
\bigoplus_{i=1}^kV_i= \mathcal{F}\otimes \mathbb{C}(y)$$ is
surjective. It follows that $\mathcal{F}$ is CGG at $y$.
\end{proof}

\begin{lemma}\label{pp}
If $\mathcal{F}$ is a non-zero $M$-regular sheaf on a complex
abelian variety $A$, then $\mathcal{F}$ has no essential base point
at any $y\in A$.
\end{lemma}
\begin{proof}
The conclusion follows from \cite[Proposition 2.13]{PP} and Lemma
\ref{lemma2.4}.
\end{proof}

\begin{lemma}\label{Ti}
Let $D$ an effective divisor on an abelian variety $A$. Take
$T_1,\ldots, T_k$ subtori of $A$ such that they generate $A$ as an
abstract group and let $\gamma_i$, $i = 1, \ldots, k$ some points of
$A$. Then $D\cap(T_i+\gamma_i)\neq\emptyset$ for at least one $i$.
\end{lemma}
\begin{proof}
See \cite[Lemma 2]{T}.
\end{proof}

\section{Proof of the main theorems}
From now on, we let $X$ be a smooth complex projective variety of
Albanese fiber dimension two and general type. Denote $a:
X\rightarrow Z=a(X)\subset A$ the Albanese map of $X$, where $A$ is
the the Albanese variety of $X$. Let $\nu: W\rightarrow Z$ be a
desingularization of the Stein factorization over $Z$. Replacing X
by an appropriate birational model, we may assume that there is a
morphism $f:X\rightarrow W$ whose general fiber is a connected
smooth surface $S$. Then we obtain a commutative diagram
\begin{eqnarray*}
\xymatrix{
X\ar[r]^{f}\ar[dr]_a & W\ar[d]^{\nu}\\
& Z}
\end{eqnarray*}
We prove the following statement, which is more general than Theorem
\ref{theorem1.1}.

\begin{theorem}\label{theorem3.1}
$|6K_X+\alpha|$ induces a birational map for any
$\alpha\in\Pic^0(X)$.
\end{theorem}
\begin{proof}
By \cite[Corollary 2.4]{CH3}, we know that $\varphi_{|6K_X+\alpha|}$
is birational for any $\alpha\in\Pic^0(X)$ if $\varphi_{|4K_S|}$ is
birational. Hence we can assume that $\varphi_{|4K_S|}$ is not
birational. It follows from \cite[Main Theorem]{Bom} that
$K_{S_0}^2=1$ and $p_g(S)=2$, where $S_0$ is the minimal model of
$S$. The theorem also tells us that $|4K_S|$ is base point free, and
$\varphi_{|4K_S|}$ is a generically double covering onto its image.
Thus the natural morphism $f^*f_*\omega_X^{\otimes4}\rightarrow
\omega_X^{\otimes4}$ defines a rational map $X\dashrightarrow
Y\subset\mathbb{P}(f_*\omega_X^{\otimes4})$ over $W$, where $Y$ is
the closure of the image.

Let $Y'\rightarrow Y$ be a resolution of singularities of $Y$ , and
let $h: X'\rightarrow Y'$ be a resolution of indeterminacies of the
corresponding rational map $X\dashrightarrow Y'$. We know that $h$
is a generically double covering branched along a reduced divisor
$B_1$. Let $\mu: \widetilde{Y}\rightarrow Y'$ be a log resolution of
$B_1$, such that $B:=\mu^*B_1-2\lfloor\frac{\mu^*B_1}{2}\rfloor$ is
smooth (see \cite[Lemma 1.3.1]{X}). We assume that $B\in|2L|$ for
some divisor $L$ on $\widetilde{Y}$. Let
$\widetilde{X}\rightarrow\widetilde{Y}$ be the double covering
branched along $B$. One sees that $\widetilde{X}$ is smooth. Thus we
obtain $K_{\widetilde{X}}=\pi^*(K_{\widetilde{Y}}+L)$. Now we have
the commutative diagram among smooth projective varieties
\begin{eqnarray*}
\xymatrix{
\widetilde{X}\ar[r]^{\pi}\ar[dr]_{\widetilde{f}} & \widetilde{Y}\ar[d]^{p}\\
& W.}
\end{eqnarray*}
We only need to show that $\varphi_{|6K_{\widetilde{X}}+\alpha|}$ is
birational for all $\alpha\in\Pic^0(\widetilde{X})$.

Let $H$ (resp. $F$) be the fiber of $\widetilde{f}$ (resp. $p$) over
a general point $w\in W$. From the construction, we know that
$\pi|_H: H\rightarrow F$ is a double covering between smooth
surfaces branched along a smooth divisor
$B|_F\in|\mathcal{O}_F(2L)|$. Hence we have
$K_H=(\pi|_H)^*(K_F+L|_F)$.

By Lemma \ref{lemma2.3}, for some $m\geq2$ and an appropriate
birational map $\sigma:\widehat{X}\rightarrow\widetilde{X}$ there
exists positive integers $b$, $c$ and there is a normal crossing
divisor $$B_m\in|bc(m-1)K_{\widehat{X}}-g^*bM|$$ such that
$\lfloor\frac{B_m}{bc}\rfloor|_{H'}\leq\mathcal{B}_{m,H'}$,
$\lfloor\frac{B_m}{bc}\rfloor\leq\mathcal{B}_{m,\alpha}$, for all
$\alpha\in\Pic^0(\widehat{X})$. Here $M$ is a given nef and big
divisor on $W$, $g$ is the composite map $\widetilde{f}\circ\sigma$,
$H'$ is the general fiber of $g$ and $b$, $c$ are sufficiently large
integers depending on $M$ and $K_{\widehat{X}}$. Thus we obtain
$$(m-1)K_{\widehat{X}}-\lfloor\frac{B_m}{bc}\rfloor\equiv\frac{1}{c}g^*M+\{\frac{B_m}{bc}\}.$$
By \cite[ Theorem 10.15]{Ko}, we know that $$H^i(A,
\nu_*g_*\mathcal{O}_{\widehat{X}}(mK_{\widehat{X}}-\lfloor\frac{B_m}{bc}\rfloor+\alpha))=H^i(W,
g_*\mathcal{O}_{\widehat{X}}(mK_{\widehat{X}}-\lfloor\frac{B_m}{bc}\rfloor+\alpha))=0$$
for all $\alpha\in\Pic^0(\widehat{X})$ and all $i>0$. It follows
from
$\sigma_*\mathcal{O}_{\widehat{X}}(mK_{\widehat{X}/\widetilde{X}})=\mathcal{O}_{\widetilde{X}}$
that
$$\mathcal{J}_m:=\sigma_*\mathcal{O}_{\widehat{X}}
(mK_{\widehat{X}/\widetilde{X}}-\lfloor\frac{B_m}{bc}\rfloor)$$ is
an ideal sheaf of $\mathcal{O}_{\widetilde{X}}$. One sees that
$\mathcal{J}_m\supset\mathcal{I}_{m,\alpha}$,
$\mathcal{J}_m|_H\supset\mathcal{I}_{m,H}$ and $$H^i(A,
\nu_*\widetilde{f}_*(\mathcal{O}_{\widetilde{X}}(mK_{\widetilde{X}}+\alpha)\otimes\mathcal{J}_m))=0$$
for all $\alpha\in\Pic^0(\widetilde{X})$ and all $i>0$, where
$\mathcal{I}_{m,\alpha}$ (resp. $\mathcal{I}_{m,H}$) is the base
ideal of $|mK_{\widetilde{X}}+\alpha|$ (resp. $|mK_H|$). In
particular,
$\nu_*\widetilde{f}_*(\mathcal{O}_{\widetilde{X}}(mK_{\widetilde{X}})\otimes\mathcal{J}_m)$
is $IT^0$.

Since
$$\pi_*\mathcal{O}_{\widetilde{X}}(6K_{\widetilde{X}})=
\mathcal{O}_{\widetilde{Y}}(6K_{\widetilde{Y}}+6L)\oplus\mathcal{O}_{\widetilde{Y}}(6K_{\widetilde{Y}}+5L),$$
we conclude that there exists ideal sheaves $\mathcal{J}_6^1$ and
$\mathcal{J}_6^2$ on $\widetilde{Y}$ such that
$$\pi_*(\mathcal{O}_{\widetilde{X}}(6K_{\widetilde{X}}\otimes\mathcal{J}_6)=
(\mathcal{O}_{\widetilde{Y}}(6K_{\widetilde{Y}}+6L)\otimes\mathcal{J}_6^1)\oplus
(\mathcal{O}_{\widetilde{Y}}(6K_{\widetilde{Y}}+5L)\otimes\mathcal{J}_6^2).$$
Hence
$\nu_*p_*(\mathcal{O}_{\widetilde{Y}}(6K_{\widetilde{Y}}+5L)\otimes\mathcal{J}_6^2)$
is $IT^0$.

On the other hand, we know that
\begin{eqnarray*}
H^0(H, \mathcal{O}_H(6K_H))&=& H^0(H,
\mathcal{O}_H(6K_H)\otimes\mathcal{J}_6|_H)\\
&=& H^0(F,\mathcal{O}_F(6K_F+6L|_F)\otimes\mathcal{J}^1_6|_F)\\
&&\oplus H^0(F,\mathcal{O}_F(6K_F+5L|_F)\otimes\mathcal{J}^2_6|_F).
\end{eqnarray*}
This implies
$$H^0(F,\mathcal{O}_F(6K_F+6L|_F))=H^0(F,\mathcal{O}_F(6K_F+6L|_F)\otimes\mathcal{J}^1_6|_F)$$
and
$$H^0(F,\mathcal{O}_F(6K_F+5L|_F))=H^0(F,\mathcal{O}_F(6K_F+5L|_F)\otimes\mathcal{J}^2_6|_F).$$
Since $|6K_H|$ is birational for a surface $H$, we obtain
$H^0(6K_F+5L|_F)\neq0$ by Lemma \ref{lemma2.2}. Thus we have
$$H^0(F,\mathcal{O}_F(6K_F+5L|_F)\otimes\mathcal{J}^2_6|_F)=
H^0(F,\mathcal{O}_F(6K_{\widetilde{Y}}+5L)\otimes\mathcal{J}^2_6|_F)\neq0.$$
This implies
$$\nu_*p_*(\mathcal{O}_{\widetilde{Y}}(6K_{\widetilde{Y}}+5L)\otimes\mathcal{J}_6^2)$$
is nonzero. Therefore we conclude that
$$\nu_*p_*(\mathcal{O}_{\widetilde{Y}}(6K_{\widetilde{Y}}+5L)\otimes\mathcal{J}_6^2)$$
is a nonzero $IT^0$ sheaf on $A$.

Hence for any $\alpha\in\Pic^0(A)$,
$H^0(\widetilde{Y},\mathcal{O}_{\widetilde{Y}}(6K_{\widetilde{Y}}+5L+\alpha))\neq0$.
By \cite[Theorem 2.8]{CH3}, one sees that
$\varphi_{|6K_{\widetilde{X}}+\alpha|}$ separates two general points
on two distinct general fibers of $\pi$. Therefore
$\varphi_{|6K_{\widetilde{X}}+\alpha|}$ is birational by Lemma
\ref{lemma2.2}.
\end{proof}

Next we study the $5$-canonical map of $X$. We take
$B\in|bc(m-1)K_X-f^*bM|$ as in Lemma \ref{lemma2.3} and
$$L_m:=(m-1)K_X-\lfloor\frac{B}{bc}\rfloor\equiv\frac{1}{c}f^*M+\{\frac{B}{bc}\},$$
$m\geq2$. One sees that $$H^0((K_X+L_m)|_S)\cong H^0(mK_S)
~\mbox{and}~ H^0(K_X+L_m+\alpha)\cong H^0(mK_X+\alpha),$$ for all
$\alpha\in\Pic^0(X)$. By \cite[Theorem 10.15]{Ko}, we have $$H^i(A,
a_*\mathcal{O}_X(K_X+L_m)\otimes\alpha)=0,$$ for all
$\alpha\in\Pic^0(A)$ and all $i>0$. In particular,
$a_*\mathcal{O}_X(K_X+L_m)$ is $IT^0$.

\begin{lemma}\label{lemma3.2}
Let $x\in X$ be a general point. Then $K_X+L_m$ has no essential
base point at $x$.
\end{lemma}
\begin{proof}
Let $F=X_{a(x)}$. By $x$ being general, we mean that
$$a_*\mathcal{O}_X(K_X+L_m)\otimes \mathbb{C}(a(x))\cong
H^0((K_X+L_m)|_F)\cong H^0(mK_F),$$ $F$ is smooth and $x$ is not a
base point of $|(K_X+L_m)|_F|$. Hence pushing forward the standard
exact sequence
$$0\rightarrow \mathcal{I}_x(K_X+L_m)\rightarrow \mathcal{O}_X(K_X+L_m)\rightarrow \mathcal{O}_x\rightarrow
0$$ to $A$, we obtain $$0\rightarrow
a_*(\mathcal{I}_x(K_X+L_m))\rightarrow
a_*\mathcal{O}_X(K_X+L_m)\rightarrow \mathcal{O}_{a(x)}\rightarrow
0.$$ Since $a_*\mathcal{O}_X(K_X+L_m)$ is $IT^0$, by Lemma \ref{pp},
$a_*\mathcal{O}_X(K_X+L_m)$ has no essential base point at $a(x)$.
It follows that $K_X+L_m$ has no essential base point at $x$.
\end{proof}

We denote $U_m$, the open subset of $\Pic^0(X)$ where
$h^0(mK_X+\alpha)$ has minimal value. Let $D_m$ be the closure of
the divisorial part of $$\mathcal{S}=\{(x,\alpha)\in X\times
U_m~|~x~\mbox{is a base point of}~|mK_X+\alpha|\}$$ in
$X\times\Pic^0(X)$. By Lemma \ref{lemma3.2}, we know that $\dim
\mathcal{S}<\dim(X\times\Pic^0(X))$. For a general $x\in X$, the
fiber of the projection $D_m\rightarrow X$ is a divisor, that we
call $D_{m,x}$.

Now we prove the following theorem, which is more general than
Theorem \ref{theorem1.2}.

\begin{theorem}\label{theorem3.3}
If, for any effective divisor $D$ of $X\times\Pic^0(X)$ which
dominates $X$, the intersection $D\cap(X\times V^0(\omega_X))$ still
dominates $X$, then $|5K_X+\alpha|$ is birational for any
$\alpha\in\Pic^0(X)$.
\end{theorem}
\begin{proof}
This theorem will be proved by three steps.

\medskip
\noindent{\bf Step 1.} {\em  Let $y\in X$ be a general point. Then
$a_*(\mathcal{I}_y(K_X+L_3))$ has no essential base point at any
point of $Z$.}

This argument has essentially been proved by Jiang in the proof of
\cite[Theorem 4.1]{J}, but we still give a proof for readers'
convenience.

Because of Lemma \ref{pp}, we want to show that
$a_*(\mathcal{I}_y(K_X+L_3))$ is $M$-regular. We have the exact
sequence $$0\rightarrow a_*(\mathcal{I}_y(K_X+L_3))\rightarrow
a_*\mathcal{O}_X(K_X+L_3)\rightarrow \mathcal{O}_{a(y)}\rightarrow
0.$$ Considering the long exact sequence obtained from the above
short exact sequence, we conclude that
$H^i(a_*(\mathcal{I}_y(K_X+L_3))\otimes\alpha)=0$ for all
$\alpha\in\Pic^0(A)$ and $i\geq2$. Hence, by $y$ being general, $y$
is a base point of $|K_X+L_3+\alpha|$ if and only if $\alpha\in
V^1(a_*(\mathcal{I}_y(K_X+L_3))$. This implies $y$ is a base point
of $|3K_X+\alpha|$ if and only if $\alpha\in
V^1(a_*(\mathcal{I}_y(K_X+L_3))$. One sees that $D_{3,y}$ is the
divisorial part of $V^1(a_*(\mathcal{I}_y(K_X+L_3))$. Thus, by
definition of $M$-regular, we know that $a_*(\mathcal{I}_y(K_X+L_3)$
is $M$-regular if $D_3=0$.

Now we can assume that $D_3\neq0$ and $pr_X: D_3\cap(X\times
V^0(\omega_X))\rightarrow X$ is dominant. By Lemma \ref{lemma3.2},
we know that $D_3$ is dominant on $X$ and $\Pic^0(X)$ via the
natural projections. Since the base locus of $|3K_X+\alpha|$ is a
proper closed subscheme of $X$ for any $\alpha\in\Pic^0(X)$, one
sees that $\bigcap_{x\in X}D_{3,x}=\emptyset$. This implies $X\times
V^0(\omega_X)\nsubseteq D_3$. Because $pr_X: D_3\cap(X\times
V^0(\omega_X))\rightarrow X$ is dominant, there exists a component
$C$ of $V^0(\omega_X)$ such that $C\nsubseteq D_{3,y}$ and $C\cap
D_{3,y}$ is not empty. By \cite[Theorem 0.1]{GL}, we can write
$C=\alpha_1+T_1$, where $T_1$ is a subtorus of
$\widehat{A}:=\Pic^0(X)$ and $\alpha_1$ is a point of $\widehat{A}$.

Let $p: \widehat{A}\rightarrow\widehat{A}/T_1$ be the quotient map.
Since $$\dim p(D_{3,y}-\alpha_1)=\dim(D_{3,y}-\alpha_1)-\dim
T_1\cap(D_{3,y}-\alpha_1)=\dim(\widehat{A}/T_1),$$ we know that
$p(D_{3,y}-\alpha_1)=\widehat{A}/T_1$. This implies that
$D_{3,y}-\alpha_1-T_1=\widehat{A}$, i.e., $D_{3,y}-C=\widehat{A}$.

Since $y\in X$ is a general point, we can choose a non-empty open
subset $U_1\subset C$ such that $y$ is not a base point of
$|K_X+\beta|$, for any $\beta\in U_1$. By considering the map
$$H^0(K_X+\alpha)\otimes H^0(2K_X+\beta)\rightarrow
H^0(3K_X+\alpha+\beta), $$ we conclude that $y$ is a base point of
$|2K_X+\beta|$ for any $\beta\in
V^1(a_*(\mathcal{I}_y(K_X+L_3)))-U_1$. By $y$ being general, we know
that $y$ is also a base point of $|K_X+L_2+\beta|$ for any $\beta\in
V^1(a_*(\mathcal{I}_y(K_X+L_3)))-U_1$.

It follows from $D_{3,y}-C=\widehat{A}$ that
$V^1(a_*(\mathcal{I}_y(K_X+L_3)))-U_1$ contains a non-empty open
subset of $\Pic^0(X)$. This contradicts that $K_X +L_2$ has no
essential base point at $y$.

\medskip
\noindent{\bf Step 2.} {\em If $|3K_S|$ is birational, then
$|5K_X+\alpha|$ induces a birational map for any
$\alpha\in\Pic^0(X)$.}

We need the following:

\begin{claim}
Let $x_1, x_2\in X$ be general points. Then $|K_X+L_3+\alpha|$
separates $x_1$, $x_2$ for general $\alpha\in \Pic^0(X)$.
\end{claim}

We follow the idea in the proof of \cite[Corollary 2.4]{CH3}. If
$x_1$ and $x_2$ are on a general $F$ of $a$. Since
$H^0((K_X+L_3)|_F)\cong H^0(3K_F)$, we know that $|(K_X+L_3)|_F|$
separates $x_1$, $x_2$. This implies
$a_*(\mathcal{I}_{x_1,x_2}(K_X+L_3))\neq
a_*(\mathcal{I}_{x_1}(K_X+L_3))$. Hence we have an exact sequence
$$0\rightarrow a_*(\mathcal{I}_{x_1,x_2}(K_X+L_3))\rightarrow
a_*(\mathcal{I}_{x_1}(K_X+L_3))\rightarrow
\mathcal{O}_{a(x_2)}\rightarrow 0.$$ If $a(x_1)\neq a(x_2)$, we
obtain $a_*(\mathcal{I}_{x_1}(K_X+L_3))\otimes
\mathbb{C}(a(x_2))\cong a_*(K_X+L_3)\otimes \mathbb{C}(a(x_2))$.
Thus we still obtain $$0\rightarrow
a_*(\mathcal{I}_{x_1,x_2}(K_X+L_3))\rightarrow
a_*(\mathcal{I}_{x_1}(K_X+L_3))\rightarrow
\mathcal{O}_{a(x_2)}\rightarrow 0.$$

By Step 1, $a_*(\mathcal{I}_{x_1}(K_X+L_3))$ has no essential base
point at $a(x_2)$. Thus for general $\alpha\in\Pic^0(X)$
$$h^0(a_*(\mathcal{I}_{x_1,x_2}(K_X+L_3))\otimes\alpha)=h^0(a_*(\mathcal{I}_{x_1}(K_X+L_3))\otimes\alpha)-1.$$
Therefore for general $\alpha\in\Pic^0(X)$
$$h^0(\mathcal{I}_{x_1,x_2}(K_X+L_3)\otimes\alpha)=h^0(\mathcal{I}_{x_1}(K_X+L_3)\otimes\alpha)-1
=h^0(\mathcal{O}_X(K_X+L_3)\otimes\alpha)-2$$ and the Claim follows.

For a fixed $\alpha\in\Pic^0(X)$, we choose general
$\beta\in\Pic^0(X)$ and consider the map
$$H^0(2K_X+\alpha-\beta)\otimes H^0(3K_X+\beta)\rightarrow
H^0(5K_X+\alpha).$$ We conclude that $|5K_X+\alpha|$ induces a
birational map for any fixed $\alpha\in\Pic^0(X)$.

\medskip
\noindent{\bf Step 3.} {\em If $|3K_S|$ is not birational, then
$|5K_X+\alpha|$ induces a birational map for any
$\alpha\in\Pic^0(X)$.}

Since $|3K_S|$ is not birational, it is well known that $S$
satisfies $(K^2_{S_0}, p_g)=(1, 2)$ or $(2, 3)$, here $S_0$ is the
minimal model of $S$. Thus $\varphi_{|3K_S|}$ is a generically
double covering onto its image (cf. \cite{Bom}). The natural
morphism $f^*f_*\omega_X^{\otimes3}\rightarrow\omega_X^{\otimes3}$
define a rational map $g:X\dashrightarrow
Y\subset\mathbb{P}(f_*\omega_X^{\otimes3})$ over $W$, where $Y$ is
the closure of the image. As the construction in the proof of
Theorem \ref{theorem3.1}, we have the commutative diagram among
smooth projective varieties
\begin{eqnarray*}
\xymatrix{
\widetilde{X}\ar[r]^{\pi}\ar[dr]_{\widetilde{f}} & \widetilde{Y}\ar[d]^{p}\\
& W,}
\end{eqnarray*}
here the double cover $\pi$ is a birational modification of $g$. The
arguments in Step 2 show that
$\varphi_{|5K_{\widetilde{X}}+\alpha|}$ separates two general points
on two distinct general fibers of $\pi$ for any
$\alpha\in\Pic^0(\widetilde{X})$. By the same method in the proof of
Theorem \ref{theorem3.1}, we can obtain Step 3. We leave the details
to the interested reader.
\end{proof}

Now we can easily prove Theorem \ref{theorem1.2}.
\begin{proof}[Proof of Theorem \ref{theorem1.2}]
If $D_3=0$, one sees that $|5K_X|$ is birational by Theorem
\ref{theorem3.3}. Hence we can assume that $D_3\neq0$. This implies
$D_{3,x}$ is an effective divisor on $\Pic^0(X)$ for general $x\in
X$.

By \cite[Theorem 0.1]{GL}, we can write
$$V^0(\omega_X)=\bigcup_{i=1}^k (T_i+\gamma_i),$$ where $T_i$'s are
subtori of $\Pic^0(X)$ and $\gamma_i$'s are some points of
$\Pic^0(X)$. It follows from Lemma \ref{Ti} that $D_{3,x}\cap
V^0(\omega_X)\neq\emptyset$, for general $x\in X$. Thus the
projection $pr_X: D_3\cap(X\times V^0(\omega_X))\rightarrow X$ is
dominant. By Theorem \ref{theorem3.3}, we obtain our conclusion.
\end{proof}

\bibliographystyle{amsplain}

\end{document}